\input ./style/arxiv-general.cfg
\documentclass[aos,MSNbibl,seceqn,nameyear,dvips]{arximspdf}
\makeatletter
   \@ifpackageloaded{graphicx}{}{\usepackage{graphicx}}
\makeatother

\doi{10.1214/15-AOS1371}
\volume{44}
\issue{3}
\pubyear{2016}
\firstpage{907}
\lastpage{927}
\docsubty{FLA}

\makeatletter

\newcommand{\rrvert}{\vert}
\newcommand{\rrVert}{\Vert}
\newcommand{\llvert}{\vert}
\newcommand{\llVert}{\Vert}
\renewcommand{\mid}{|}
\newcommand{\overset}{\stackrel}
\newcommand{\eqref}[1]{(\ref{#1})}
\newtheorem{theorem}{Theorem}[section]
\newtheorem{corollary}[theorem]{Corollary}
\newtheorem{lemma}[theorem]{Lemma}
\newtheorem{proposition}[theorem]{Proposition}
\newcommand{\ones}{\mathbf1}
\newcommand{\argmin}{\mathop{\arg\min}\limits}
\newcommand{\Expect}{\mathbf{E}}
\newcommand{\sign}{\operatorname{sign}}

\newcommand{\diag}{\operatorname{diag}}
\newcommand{\unif}{\operatorname{Unif}}
\renewcommand{\dagger}{+}

\newcommand{\R}{\mathbb{R}}
\newcommand{\V}{{\mathcal V}}
\newcommand{\Ee}{\mathbb{E}}
\newcommand{\Pp}{\mathbb{P}}
\newcommand{\E}{M}
\newcommand{\hE}{\hat{M}}
\newcommand{\bb}{{\mathbf b}}
\newcommand{\bc}{{\mathbf c}}
\newcommand{\be}{{\mathbf e}}
\newcommand{\bs}{{\mathbf s}}
\newcommand{\bu}{{\mathbf u}}
\newcommand{\bw}{{\mathbf w}}
\newcommand{\bx}{{\mathbf x}}
\newcommand{\by}{{\mathbf y}}
\newcommand{\bz}{{\mathbf z}}
\newcommand{\mathbfu}{{\bolds\mu}}
\newcommand{\bbeta}{{\bolds\beta}}
\newcommand{\hbeta}{\hat{\beta}}
\newcommand{\boldeta}{{\bolds\eta}}
\makeatother

\begin{document}
\begin{frontmatter}

\title{Exact post-selection inference, with application to~the~lasso}
\runtitle{Post-selection inference for the lasso}

\begin{aug}
\author[A]{\fnms{Jason D.} \snm{Lee}\thanksref{M1,T1}\ead[label=e1]{jasondlee@berkeley.edu}},
\author[B]{\fnms{Dennis L.} \snm{Sun}\corref{}\thanksref{M2,T2}\ead[label=e2]{dennisliusun@gmail.com}},
\author[C]{\fnms{Yuekai} \snm{Sun}\thanksref{M1,T3}\ead[label=e3]{yuekai@berkeley.edu}}\\
\and
\author[D]{\fnms{Jonathan E.} \snm{Taylor}\thanksref{M3,T4}\ead[label=e4]{jonathan.taylor@stanford.edu}}
\runauthor{Lee, Sun, Sun and Taylor}
\affiliation{University of California, Berkeley\thanksmark{M1},
California Polytechnic State University\thanksmark{M2} and
Stanford University\thanksmark{M3}}
%
\address[A]{J. D. Lee\\
University of California, Berkeley\\
465 Soda Hall\\
Berkeley, California 94720\\
USA\\
\printead{e1}}
\address[B]{D. L. Sun\\
California Polytechnic State University\\
Building 25, Room 107D\\
San Luis Obispo, California\\
USA\\
\printead{e2}}
\address[C]{Y. Sun\\
University of California, Berkeley\\
367 Evans Hall\\
Berkeley, California 94720\\
USA\\
\printead{e3}}
\address[D]{J. E. Taylor\\
Stanford University\\
390 Serra Mall\\
Stanford, California\\
USA\\
\printead{e4}\hspace*{38pt}}
\end{aug}
\thankstext{T1}{Supported by a National Defense Science and Engineering Graduate Fellowship and a Stanford Graduate Fellowship.}
\thankstext{T2}{Supported by a Ric Weiland Graduate Fellowship and the Stanford Genome Training Program (SGTP; NIH/NHGRI).}
\thankstext{T3}{Supported in part by the NIH Grant U01GM102098.}
\thankstext{T4}{Supported by the NSF Grant DMS-12-08857, and by AFOSR Grant 113039.}

%
\received{\smonth{1} \syear{2015}}
%
\revised{\smonth{9} \syear{2015}}

%
\begin{abstract}
We develop a general approach to valid inference after model selection.
At the core of our framework is a result that characterizes the
distribution of a post-selection estimator
conditioned on the \emph{selection event}.
We specialize the approach to model selection by the lasso to form
valid confidence intervals for the selected
coefficients and test whether all relevant variables have been included
in the model.
\end{abstract}

%
\begin{keyword}[class=AMS]
\kwd[Primary ]{62F03}
\kwd{62J07}
\kwd[; secondary ]{62E15}
\end{keyword}
\begin{keyword}
\kwd{Lasso}
\kwd{confidence interval}
\kwd{hypothesis test}
\kwd{model selection}
\end{keyword}
\end{frontmatter}

\setcounter{footnote}{4}

\section{Introduction}
\label{sec:intro}

As a statistical technique, linear regression is both simple and
powerful. Not only does it provide estimates of the ``effect'' of each
variable, but it also quantifies the uncertainty in those estimates,
allowing inferences to be made about the effects. However, in many
applications, a practitioner starts with a large pool of candidate
variables, such as genes or demographic features, and does not know
{a priori} which are relevant. This is especially problematic
when there are more variables than observations, since then the model is unidentifiable (at least in the setting where the predictors are assumed fixed).

In such settings, it is tempting to let the data decide which
variables to include in the model. For example, one common approach
when the number of variables is not too large is to fit a linear model
with all variables included, observe which ones are significant at
level $\alpha$, and then refit the linear model with only those
variables included. The problem with this is that the $p$-values can no
longer be trusted, since the variables that are selected will tend to
be those that are significant. Intuitively, we are ``overfitting'' to a
particular realization of the data.

To formalize the problem, consider the standard linear regression
setup, where the response $\by \in \R^n$ is generated from a multivariate normal distribution:
%
\begin{equation}
\by\sim N\bigl(\mathbfu, \sigma^2I_n\bigr) \label{eq:model}
\end{equation}
where $\bolds{\mu}$ is modeled as a linear function of predictors $\bx_1,
\ldots, \bx_p \in\R^n$, and $\sigma^2$ is assumed known. (We
consider the more realistic case where $\sigma^2$ is unknown in
Section~\ref{sec:estimate:sigma}.) We choose a subset $M \subset\{1,
\ldots, p\}$ and ask for the linear combination of the predictors in
$M$ that minimizes the expected error, that is,
%
\begin{equation}
\bbeta^M \equiv\mathop{\argmin}\limits_{{\mathbf b}^M} \Ee\bigl\llVert
\by-
X_M {\mathbf b}^M \bigr\rrVert^2 =
X_M^\dagger\mathbfu, \label{eq:target}
\end{equation}
where $X_M^\dagger\equiv(X_M^T X_M)^{-1} X_M^T$ is the pseudo-inverse
of $X_M$. Notice that \eqref{eq:target} implies that the targets
$\beta^M_j$ and $\beta^{M'}_j$ in different models $M \neq M'$ are in
general different. This is simply a restatement of the well-known fact
that a regression coefficient describes the effect of a predictor, {\em
adjusting for the other predictors in the model}. In general, the
coefficient of a predictor cannot be compared across different models.

Thus, ``inference after selection'' is ambiguous in linear regression
because the target of inference changes with the selected model [\citet
{berk2013posi}]. In the next section, we discuss several ways to resolve
this ambiguity.

\section{Post-selection inference in linear regression}
\label{sec:goals}

At first blush, the fact that the target $\bbeta^M$ changes with the
model is deeply troubling, since it seems to imply that the parameters
are random. However, the randomness is actually in the {\em choice} of
which parameters to consider, not in the parameters themselves. Imagine
that there are {a priori} $p2^{p-1}$ well-defined population
parameters, one for each coefficient in all $2^p$ possible models:
\[
\bigl\{ \beta^M_j: M \subset\{1, \ldots, p\}, j \in M
\bigr\}.
\]
We\vspace*{1pt} only ever form inferences for the parameters $\beta^{\hat M}_j$ in
the model $\hat M$ we select. This adaptive choice of which parameters
to consider can lead to inferences with undesirable frequency
properties, as noted by
\citet{benjamini2005false} and \citet{benjamini2009selective}.

To be\vspace*{-3pt} concrete, suppose we want a confidence interval $C^{\hat M}_j$
for a parameter $\beta^{\hat M}_j$. What frequency properties should
$C^{\hat M}_j$ have? By analogy to the classical setting, we might require that
\[
\Pp\bigl(\beta^{\hat M}_j \in C^{\hat M}_j
\bigr) \geq1-\alpha,
\]
but the event inside the probability is not well-defined because $\beta^M_j$ is undefined when
$j \notin M$. Two ways around this issue are suggested by \citet
{berk2013posi}:
\begin{longlist}[2.]
\item[1.] \textit{Conditional coverage}: Since we form an interval for $\beta^M_j$
if and only if model $M$ is selected, that is, $\hat M = M$, it makes
sense to condition on this event. Hence,\vadjust{\goodbreak} we might require that our
confidence interval $C^M_j$ satisfy
%
\begin{equation}
\Pp\bigl(\beta^M_j \in C^M_j
\mid\hat M = M\bigr) \geq1-\alpha. \label{eq:conditional_coverage}
\end{equation}
The benefit of this approach is that we avoid ever having to compare
coefficients across two different models $M \neq M'$.

Another way to understand conditioning on the model is to consider {\em
data splitting} [\citet{cox1975note}], an approach to post-selection
inference that most statisticians would agree is valid. In data
splitting, the data is divided into two halves, with one half used to
select the model and the other used to conduct inference. 
\citet{fithian2014optimal} argues that inferences obtained by
data splitting are only valid conditional on the model that was selected
on the first half of the data. Therefore, conditional coverage is a
reasonable frequency property to require of a post-selection confidence interval.


\item[2.] \textit{Simultaneous coverage}: It\vspace*{1pt} also makes sense to talk about events
that are defined simultaneously over all $j \in\hat M$. \citet
{berk2013posi} propose controlling the familywise error rate
%
\begin{equation}
\mathrm{FWER} \equiv\Pp\bigl(\beta^{\hat M}_j \notin
C^{\hat M}_j \mbox{ for any }j \in\hat M\bigr), \label{eq:fwer}
\end{equation}
but this is very stringent when many predictors are involved.

Instead of controlling the probability of making {\em any} error, we
can control the expected proportion of errors---although ``proportion
of errors'' is ambiguous in the event that we select zero
variables. \citet{benjamini2005false} simply declare the error to
be zero when $\llvert \hat M \rrvert = 0$:
%
\begin{equation}
\mathrm{FCR} \equiv\Ee\biggl[ \frac{\llvert \{ j \in\hat M: \beta
^{\hat M}_j
\notin C^{\hat M}_j \}\rrvert }{ \llvert \hat M \rrvert }; \llvert
\hat M \rrvert> 0 \biggr],
\label{eq:fcr}
\end{equation}
while \citet{storey2003positive} suggests conditioning on
$\llvert \hat M \rrvert > 0$:
%
\begin{equation}
\mathrm{pFCR} \equiv\Ee\biggl[ \frac{\llvert \{ j \in\hat M:
\beta
^{\hat M}_j \notin C^{\hat M}_j \}\rrvert }{ \llvert \hat M
\rrvert } \Big| |\hat M |
> 0 \biggr]. \label{eq:pfcr}
\end{equation}
The two criteria are closely related. Since $\mathrm{FCR} = \mathrm{pFCR} \cdot\Pp
(\llvert \hat M \rrvert > 0)$, $\mathrm{pFCR}$ control implies $\mathrm{FCR}$ control.
\end{longlist}

The two ways above are related: conditional coverage \eqref
{eq:conditional_coverage} implies $\mathrm{pFCR}$ \eqref{eq:pfcr} (and hence,
FCR) control.

%
\begin{lemma}\label{lem:fcr}
Consider a family of intervals $\{ C^{\hat M}_j \}_{j\in\hat M}$ that
each have conditional $(1-\alpha)$ coverage:
\[
\Pp\bigl(\beta^{\hat M}_j \notin C^{\hat M}_j
\mid\hat M = M \bigr) \leq\alpha\qquad\mbox{for all $M$ and $j \in M$}.
\]
Then $\mathrm{FCR} \leq \mathrm{pFCR} \leq\alpha$.
\end{lemma}

\begin{pf}
Condition on $\hat M$ and iterate expectations:
\begin{eqnarray*}
\mathrm{pFCR} &=& \Ee\biggl[\Ee\biggl[ \frac{\llvert \{ j
\in
\hat M: \beta^{\hat M}_j \notin C^{\hat M}_j \}\rrvert }{\llvert \hat
M \rrvert } \Big|\hat M
\biggr] \Big|\llvert\hat M \rrvert> 0 \biggr]
\\
&=& \Ee\biggl[
\frac{\sum_{j\in\hat M} \Pp( \beta
^{\hat M}_j \notin C^{\hat M}_j \mid\hat M )}{\llvert \hat
M \rrvert } \Big|\llvert\hat M \rrvert> 0 \biggr]
\\
&\leq&\Ee\biggl[\frac{\alpha\llvert \hat M \rrvert
}{\llvert \hat M \rrvert } \Big|\llvert\hat M \rrvert> 0
\biggr]
\\
&=& \alpha.
\end{eqnarray*}\upqed
\end{pf}

Theorem 2 in \citet{selectionbenjamini} proves a special case of
Lemma~\ref{lem:fcr} for a particular selection procedure, and
Proposition 11 in \citet{fithian2014optimal} provides a more
general result, but this result is sufficient for our purposes:
to establish that conditional coverage is a sensible criterion to consider in post-selection inference.


Although the criterion is easy to state, how do we construct an
interval with conditional
coverage? This requires that we understand the conditional distribution
\[
\by\mid\bigl\{\hat{M}(\by) = M\bigr\}, \qquad\by\sim N\bigl(\mu,
\sigma^2 I\bigr). %
\]
One of the main contributions of this paper is to show that this distribution is
indeed possible to characterize, making valid post-selection inference feasible in the context of linear regression.

\section{Outline of our approach}

We have argued that post-selection intervals for regression
coefficients should have $1-\alpha$ coverage conditional on the
selected model:
\[
\Pp\bigl(\beta^M_j \in C^M_j
\mid\hat M = M\bigr) \geq1-\alpha,
\]
both because this criterion is interesting in its own right and because
it implies FCR control. To obtain an interval with this property, we
study the conditional distribution
%
\begin{equation}
\boldeta_M^T \by\mid\{\hat M = M\}, \label{eq:cond_dist}
\end{equation}
which will allow, more generally, conditional inference for parameters
of the form $\boldeta_M^T \mathbfu$. In particular, the regression
coefficients $\beta^M_j = \be_j^T X_M^\dagger\mathbfu$ can be written
in this form, as can many other linear contrasts.

Our\vspace*{1pt} paper focuses on the specific case where the lasso is used to
select the model~$\hat M$. We begin in Section~\ref
{sec:lasso-selection} by characterizing the event $\{ \hat M = M \}$
for the lasso. As it turns out, this event is a union of polyhedra.
More precisely, the event $\{ \hat M = M, \hat\bs_M = \bs_M \}$,
that specifies the model {\em and} the signs of the selected variables,
is a polyhedron of the form
\[
\bigl\{\by\in\R^n: A(M, \bs_M) \by\leq\bb(M,
\bs_M) \bigr\}.
\]

Therefore, if we condition on both the model and the signs, then we
only need to study
%
\begin{equation}
\boldeta^T \by\mid\{A\by\leq\bb\}.
\end{equation}
We do this in Section~\ref{sec:polyhedra}. It turns out that this
conditional distribution is essentially a (univariate) truncated
Gaussian. We use this to derive a statistic $F^\bz(\boldeta^T \by)$
whose distribution given $\{A\by\leq\bb\}$ is $\unif(0, 1)$.

\subsection{Related work}

The resulting post-selection test has a similar structure to the
pathwise significance tests of \citet{lockhart2012significance} and
\citet{taylor2014post}, which also are conditional tests. However, the
intended application of our test is different. While their significance
tests are specifically intended for the path context, our framework
allows more general questions about the model the lasso selects: we can
test the model at any value of $\lambda$ or form confidence intervals
for an individual coefficient in the model.

There is also a parallel literature on confidence intervals for
coefficients in high-dimensional linear models based on the lasso
estimator [\citet
{van2013asymptotically,zhang2011confidence,javanmard2013confidence}].
The difference between their work and ours is that they do not address
post-selection inference; their target is $\bbeta^0$, the coefficients
in the true model, rather than $\bbeta^{\hat M}$, the coefficients in
the selected model. The two will not be the same unless $\hat M$
happens to contain all nonzero coefficients of $\bbeta^0$. Although
inference for $\bbeta^0$ is appealing, it requires assumptions about
correctness of the linear model and sparsity of $\bbeta^0$. \citet
{potscher2010confidence} consider confidence intervals for the
hard-thresholding and soft-thresholding estimators in the case of
orthogonal design. Our approach instead regards the selected model as a
linear approximation to the truth, a view shared by \citet
{berk2013posi} and \citet{miller2002subset}.

The idea of post-selection inference conditional on the selected model
appears in \citet{potscher1991effects}, although the notion of
inference conditional on certain \emph{relevant subsets} dates back to
\citet{fisher1956test}; see also \citet{robinson1979conditional}.
\citeauthor{leeb2005model} (\citeyear{leeb2005model,leeb2006can}) obtained a number of negative
results about estimating the distribution of a post-selection
estimator, although they note their results do not necessarily preclude
the possibility of post-selection inference.
\citet{benjamini2005false} also consider conditioning
on the selection event, although they argue that this is too
conservative. To the contrary, we show that conditioning on the
selected model can produce reasonable confidence intervals in a wide
variety of situations.


Inference conditional on selection has also appeared in literature on
the {\em winner's curse}: \citet
{sampson2005drop,sill2009drop,zhong2008bias,zollner2007overcoming}.
These works
are not really associated with model selection in linear regression,
though they employ a similar
approach to inference.


\section{The lasso and its selection event}
\label{sec:lasso-selection}

In this paper, we apply our post-selection inference procedure to the
model selected by the lasso [\citet{tibshiranilasso}]. The lasso
estimate is the solution to the usual least squares problem with an
additional $\ell_1$ penalty on the coefficients:
%
\begin{equation}
\label{eq:lasso} \hat\bbeta\in\mathop{\argmin}\limits_{\bbeta} \frac
{1}{2} \llVert
\by-X\bbeta\rrVert^2_2+ \lambda\llVert\bbeta\rrVert
_1.
\end{equation}
The $\ell_1$ penalty shrinks many of the coefficients to exactly zero,
and the tradeoff between sparsity and fit to the data is controlled\vspace*{1pt} by
the penalty parameter \mbox{$\lambda\geq0$}. However, the distribution of
the lasso estimator $\hbeta$ is known only in the less interesting $n
\gg p$ case \citet{knight2000lasso}, and even then, only
asymptotically. Inference based on the lasso estimator is still an open
question.

Because the lasso produces sparse solutions, we can define model
``selected'' by the lasso to be simply the set of predictors with
nonzero coefficients:
\[
\hat M = \{ j: \hat\beta_j \neq0 \}.
\]
Then post-selection inference seeks to make inferences about $\bbeta
^M$, given $\{\hat M = M\}$, as defined in \eqref{eq:target}.

The rest of this section focuses on characterizing this event $\{ \hat
M = M \}$. We begin by noting that in order for a vector of
coefficients $\hat\bbeta$ and a vector of signs $\hat\bs$ to be
solutions to the lasso problem \eqref{eq:lasso}, it is necessary and
sufficient that they satisfy the Karush--Kuhn--Tucker (KKT) conditions:
%
\begin{eqnarray}
&& \hspace*{-10pt}X^T(X\hat\bbeta- \by) + \lambda\hat\bs= 0, \label{eq:lasso-kkt}
\\
\hat s_i &=& \sign(\hat\beta_j) \qquad\mbox{if $\hat \beta_j \neq0$,}
\nonumber\\[-8pt]\\[-8pt]\nonumber
\hat s_i &\in& [-1, 1] \qquad\mbox{if $\hat\beta_j = 0$.}
\end{eqnarray}

Following \citet{tibshirani2013lasso}, we consider the {\em
equicorrelation set}
%
\begin{equation}
\label{eq:equicor} \hat M \equiv\bigl\{i\in\{1,\dots,p\}: \llvert\hat
s_i\rrvert= 1 \bigr\}.
\end{equation}
%
Notice that we have implicitly identified the model $\hat M$ with the equicorrelation set.
Since $\llvert \hat s_i\rrvert = 1$ for
any $\hat\beta_i \neq
0$, the equicorrelation set does in fact contain all predictors with
nonzero coefficients, although it may also include some predictors with
zero coefficients. However, for almost every $\lambda$, the
equicorrelation set is precisely the set of predictors with nonzero
coefficients.


It turns out that it is easier to first characterize $\{ (\hat\E,
\hat\bs) = (\E, \bs) \}$ and obtain $\{ \hat\E= \E\}$ as a
corollary by taking a union over the possible signs. The next result is
an important first step.

%
\begin{lemma}
\label{lem:equiv_sets}
Assume the columns of $X$ are in general position [\citet
{tibshirani2013lasso}]. Let $\E\subset\{1, \dots, p\}$ and $\bs\in\{
-1,1\}^{\llvert \E\rrvert }$ be a candidate set of
variables and their signs,
respectively. Define the random variables
%
\begin{eqnarray}
\bw(\E, \bs) &:=& \bigl(X_\E^T X_\E
\bigr)^{-1} \bigl(X_\E^T \by- \lambda\bs\bigr),
\label{eq:beta-active}
\\
\bu(\E, \bs) &:=& X_{-\E}^T\bigl(X_\E^T
\bigr)^\dagger\bs+ \frac{1}{\lambda
} X_{-\E}^T (I -
P_{\E}) \by,\label{eq:inactive_subgradient}
\end{eqnarray}
where\vspace*{1pt} $P_\E\equiv X_\E(X_\E^T X_\E)^{-1} X_\E$ is projection onto
the column span of $X_\E$. Then the selection procedure can be
rewritten in terms of $\bw$ and $\bu$ as
%
\begin{eqnarray}
\bigl\{ (\hat\E, \hat\bs) &=& (\E, \bs) \bigr\} = \bigl\{ \sign\bigl(\bw
(\E, \bs)
\bigr) = \bs, \bigl\llVert\bu(\E, \bs)\bigr\rrVert_\infty< 1 \bigr\}.
\label{eq:equiv_sets}
\end{eqnarray}
\end{lemma}

\begin{pf}
First, we rewrite the KKT conditions \eqref{eq:lasso-kkt} by
partitioning them according to the equicorrelation set $\hat\E$,
adopting the convention that $-\hat\E$ means ``variables not in $\hat
\E$'':
\begin{eqnarray*}
X_{\hat\E}^T (X_{\hat\E}\hat{\bbeta}_{\hat\E} -
\by) + \lambda\hat\bs_{\hat\E} &=& 0,
\\
X_{-\hat\E}^T (X_{\hat\E}\hat{\bbeta}_{\hat\E} -
\by) + \lambda\hat\bs_{-\hat\E} &=& 0,
\\
\sign(\hat\bbeta_{\hat\E}) &=& \hat\bs_{\hat M},
\\
\llVert\hat\bs_{-\hat\E}\rrVert_\infty&<& 1.
\end{eqnarray*}
Since the KKT conditions are necessary and sufficient for a solution,
we obtain that ${\{ (\hat\E, \hat\bs) = (\E,\bs) \}}$ if and only
if there exist $\bw$ and $\bu$ satisfying
\begin{eqnarray*}
X_\E^T (X_\E\bw- \by) + \lambda\bs&=& 0,
\\
X_{-\E}^T (X_\E\bw- \by) + \lambda\bu&=& 0,
\\
\sign(\bw) &=& \bs,
\\
\llVert\bu\rrVert_\infty&<& 1.
\end{eqnarray*}

We can solve the first two equations for $\bw$ and $\bu$ to obtain
the equivalent set of conditions
\begin{eqnarray*}
\bw&=& \bigl(X_\E^T X_\E\bigr)^{-1}
\bigl(X_\E^T \by- \lambda\bs\bigr),
\\
\bu&=& X_{-\E}^T\bigl(X_\E^T
\bigr)^\dagger\bs+ \frac{1}{\lambda} X_{-\E}^T (I -
P_{\E}) \by,
\\
\sign(\bw) &=& \bs,
\\
\llVert\bu\rrVert_\infty&<& 1,
\end{eqnarray*}
where the first two are the definitions of $\bw$ and $\bu$ given in
\eqref{eq:beta-active} and \eqref{eq:inactive_subgradient}, and the
last two are the conditions on $\bw$ and $\bu$ given in \eqref
{eq:equiv_sets}.
\end{pf}

Lemma \ref{lem:equiv_sets} is remarkable because it says that the
event ${\{ (\hat\E, \hat\bs) = (\E, \bs) \}}$ can be rewritten as
affine constraints on $\by$. This is because $\bw$ and $\bu$ are
already affine functions of $\by$, and the constraints $\sign(\cdot)
= \bs$ and $\llVert\cdot\rrVert_\infty< 1$ can also be
rewritten in terms of
affine constraints. The following proposition makes this explicit.

%
\begin{proposition}
\label{prop:A_b}
Let $\bw$ and $\bu$ be defined as in \eqref{eq:beta-active} and
\eqref{eq:inactive_subgradient}. Then
%
\begin{eqnarray}
\bigl\{ \sign(\bw) = \bs, \llVert\bu\rrVert_\infty< 1 \bigr\} &=&
\biggl\{ \pmatrix{ A_0(\E, \bs)
\vspace{3pt}\cr
A_1(\E, \bs)} \by<
\pmatrix{ \bb_0(\E, \bs)
\vspace{3pt}\cr
\bb_1(\E, \bs) } \biggr\},
\label{eq:polyhedron}
\end{eqnarray}
where $A_0, \bb_0$ encode the ``inactive'' constraints $\{\llVert
\bu\rrVert_\infty< 1\}$, and $A_1, \bb_1$ encode the
``active'' constraints $\{
\sign(\bw) = \bs\}$. These matrices have the explicit forms
\begin{eqnarray*}
A_0(\E, \bs) &=& \frac{1}{\lambda}\pmatrix{X_{-\E}^T
(I - P_\E)
\vspace{3pt}\cr
- X_{-\E}^T (I - P_\E)
},
\\[3pt]
\bb_0(\E, \bs) &=& \pmatrix{\ones- X_{-\E}^T
\bigl(X_\E^T\bigr)^\dagger\bs
\vspace{3pt}\cr
\ones+
X_{-\E}^T \bigl(X_\E^T
\bigr)^\dagger\bs},
\\[3pt]
A_1(\E, \bs) &=& -\diag(\bs) \bigl(X_\E^T
X_\E\bigr)^{-1} X_\E^T,
\\[3pt]
\bb_1(\E, s) &=& -\lambda\diag(\bs) \bigl(X_\E^T
X_\E\bigr)^{-1} \bs.
\end{eqnarray*}
\end{proposition}

\begin{pf}
First, substituting expression \eqref{eq:beta-active} for $\bw$, we
rewrite the ``active'' constraints as
\begin{eqnarray*}
\bigl\{ \sign(\bw) = \bs\bigr\} &=& \bigl\{ \diag(\bs) \bw> 0 \bigr\}
\\
&=& \bigl\{ \diag(\bs) \bigl(X_\E^T X_\E
\bigr)^{-1}\bigl(X_\E^T \by- \lambda\bs\bigr) > 0
\bigr\}
\\
&=& \bigl\{ A_1(\E, \bs) \by< \bb_1(\E, \bs) \bigr\}.
\end{eqnarray*}

Next, substituting expression \eqref{eq:inactive_subgradient} for $\bu
$, we rewrite the ``inactive'' constraints as
\begin{eqnarray*}
\bigl\{ \llVert\bu\rrVert_\infty< 1 \bigr\} &=& \biggl\{ -\ones<
X_{-\E}^T\bigl(X_\E^T
\bigr)^\dagger\bs+ \frac{1}{\lambda} X_{-\E}^T (I -
P_{\E}) \by< \ones\biggr\}
\\
&=& \bigl\{ A_0(\E, \bs) \by< \bb_0(\E, \bs) \bigr\}.
\end{eqnarray*}\upqed
\end{pf}

Combining Lemma \ref{lem:equiv_sets} with Proposition \ref{prop:A_b},
we obtain the following.

%
\begin{theorem}\label{thm:lasso_partition}
Let\vspace*{2pt} $A(\E, \bs) = { A_0(\E, \bs) \choose
A_1(\E, \bs) }$ and $b(\E, \bs) = { \bb_0(\E, \bs) \choose
\bb_1(\E, \bs) }$, where $A_i$ and $b_i$ are defined in Proposition
\ref{prop:A_b}. Then
\[
\displaystyle\{ \hat\E= \E, \hat\bs= \bs\} = \bigl\{ A(\E, \bs)\by\leq
\bb(\E,
\bs) \bigr\}.
\]
\end{theorem}

As a corollary, $\{ \hat\E= \E\}$ is simply the union of the above
events over all possible sign patterns.

%
\begin{corollary}
$\{ \hat\E= \E\} = \bigcup_{\bs\in\{-1, 1\}^{\llvert \E
\rrvert }}\{ A(\E, \bs)\by\leq\bb(\E, \bs) \}$.
\label{cor:lasso_partition}
\end{corollary}

Figure~\ref{fig:lasso_partition} illustrates Theorem~\ref
{thm:lasso_partition} and Corollary~\ref{cor:lasso_partition}. The
lasso partitions of $\R^n$ into polyhedra according to the model it
selects and the signs of the coefficients. The shaded area corresponds
to the event $\{\hat M = \{1, 3\}\}$, which is a union of two
polyhedra. Notice that the sign patterns $\{+, -\}$ and $\{-, +\}$ are
not possible for the model $\{1, 3\}$.

%
\begin{figure}

\includegraphics{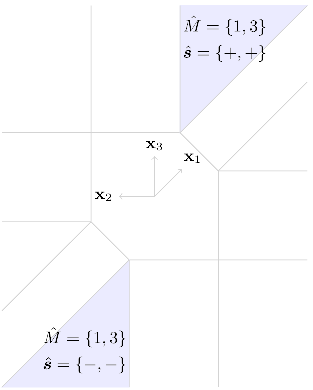}

\caption{A geometric picture illustrating Theorem \protect\ref
{thm:lasso_partition} for $n=2$ and $p = 3$. The lasso partitions $\R
^n$ into polyhedra according to the selected model and signs.}
\label{fig:lasso_partition}
\end{figure}


\section{Polyhedral conditioning sets}
\label{sec:polyhedra}

In order to obtain inference conditional on the model, we need to
understand the distribution of
\[
\boldeta_\E^T \by\mid\{ \hat\E= \E\}.
\]
However, as we saw in the previous section, $\{ \hat\E= \E\}$ is a
union of polyhedra, so it is easier to condition on both the model {\em
and the signs},
\[
\boldeta_\E^T \by\mid\{ \hat\E= \E, \hat\bs= \bs\},
\]
since the conditioning event is a single polyhedron $\{ A(\E, \bs)\by
\leq\bb(\E, \bs)\}$. Notice that inferences that are valid
conditional on this finer event will also be valid conditional on $\{
\hat\E= \E\}$. For example, if a confidence interval $C^\E_j$ for
$\beta^\E_j$ has $(1-\alpha)$ coverage conditional on the model and signs
\[
\Pp\bigl(\beta^\E_j \in C^\E_j
\mid\hat\E= \E, \hat\bs= \bs\bigr) \geq1-\alpha,
\]
it will also have $(1-\alpha)$ coverage conditional only on the model by the Law of Total Probability:
\begin{eqnarray*}
\Pp\bigl(\beta^\E_j \in C^\E_j
\mid\hat\E= \E\bigr) &=& \sum_{\bs} \Pp\bigl(
\beta^\E_j \in C^\E_j\mid\hat
\E= \E, \hat\bs= \bs\bigr) \Pp(\hat\bs= \bs\mid\hat\E= \E)
\\[-2pt]
&\geq&\sum_{\bs} (1-\alpha) \Pp(\hat\bs= \bs\mid\hat
\E= \E)
\\[-2pt]
&=& 1-\alpha.
\end{eqnarray*}

This section is divided into two subsections. First, we study how to
condition on a single polyhedron; this will allow us to condition on
$\{ \hat\E= \E, \hat\bs= \bs\}$. Then we extend
the framework to condition on a union of polyhedra, which will allow us
to condition only on the model $\{ \hat\E= \E\}$. The inferences
obtained by conditioning on the model will in general be more efficient
(i.e., narrower intervals, more powerful tests), at the price of more
computation.

\subsection{Conditioning on a single polyhedron}

Suppose we observe $\by\sim N(\mathbfu,\break  \Sigma)$, and $\boldeta\in
\R
^n$ is some direction of interest. To understand the distribution of
%
\begin{equation}
\boldeta^T \by\mid\{ A\by\leq\bb\},
\end{equation}
we rewrite $\{ A\by\leq\bb\}$ in terms of $\boldeta^T \by$ and a
component $\bz$ which is independent of $\boldeta^T \by$. That
component is
%
\begin{equation}
\bz\equiv\bigl(I_n - \bc\boldeta^T\bigr)\by,
\label{eq:z}
\end{equation}
where
%
\begin{equation}
\bc\equiv\Sigma\boldeta\bigl(\boldeta^T \Sigma\boldeta
\bigr)^{-1}. \label{eq:c}
\end{equation}
It is easy to verify that $\bz$ is uncorrelated with, and hence
independent of, $\boldeta^T \by$.
Notice that in the case where $\Sigma= \sigma^2 I_n$, $\bz $ is simply the
residual $(I_n - P_{\boldeta})\by$ from projecting $\by$ onto
$\boldeta$.


We can now rewrite $\{A\by\leq\bb\}$ in terms of $\boldeta^T \by$
and $\bz$.


%
\begin{lemma}
\label{lem:conditional}
Let $\bz$ be defined as in \eqref{eq:z} and $\bc$ as in \eqref
{eq:c}. Then the conditioning set can be rewritten as follows:
\[
\{A\by\leq\bb\} = \bigl\{\V^-(\bz) \leq\boldeta^T \by\leq\V^+(\bz),
\V^0(\bz) \geq0 \bigr\},
\]
where
%
\begin{eqnarray}
\V^-(\bz) &\equiv&\max_{j: (A\bc)_j < 0} \frac{b_j - (A\bz)_j}{
(A\bc)_j},
\label{eq:v_minus}
\\[-2pt]
\V^+(\bz) &\equiv&\min_{j: (A\bc)_j > 0} \frac{b_j - (A\bz
)_j}{(A\bc)_j}, \label{eq:v_plus}
\\[-2pt]
\V^0(\bz) &\equiv&\min_{j: (A\bc)_j = 0} b_j - (A
\bz)_j. \label{eq:v_zero}
\end{eqnarray}
Note that $\V^-$, $\V^+$, and $\V^0$ refer to functions. Since they
are functions of $\bz$ only, \eqref{eq:v_minus}--\eqref{eq:v_zero}
are independent of $\boldeta^T \by$.\vadjust{\goodbreak}
\end{lemma}

%
\begin{figure}

\includegraphics{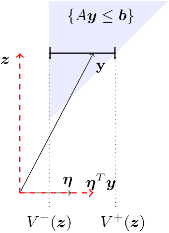}

%
%
%
%
%
%
%
%
\caption{A geometric interpretation of why the event $ \{A\by
\leq\bb\}$ can be characterized as $\{ \V^-(\bz) \leq \bolds{\eta}^{T} \mathbf{y} \leq\V^+(\bz)\}$. Assuming $\Sigma= I$ and $\llVert\bolds{\eta}
\rrVert_2 = 1$,
$\V^-(\bz)$ and $\V^+(\bz)$ are functions of $\bz$ only, which is
independent of $\boldeta^T \by$.}
\label{fig:polyhedron}
\end{figure}

\begin{pf}
We can decompose $\by= \bc(\boldeta^T \by) + \bz$ and rewrite the
polyhedron as
\begin{eqnarray*}
\{ A \by\leq\bb\} &=& \bigl\{ A\bigl(\bc\bigl(\boldeta^T \by\bigr) +
\bz\bigr) \leq\bb\bigr\}
\\
&=& \bigl\{ A\bc\bigl(\boldeta^T \by\bigr) \leq\bb- A \bz\bigr\}
\\
&=& \bigl\{ (A\bc)_j \bigl(\boldeta^T \by\bigr) \leq
b_j - (A\bz)_j \mbox{ for all $j$} \bigr\}
\\
&=& \left.\cases{
\displaystyle\boldeta^T \by\leq\frac{b_j - (A\bz)_j}{(A\bc)_j}, &\quad for $j: (A\bc)_j > 0$,
\vspace*{3pt}\cr
\displaystyle \boldeta^T \by\geq \frac{b_j - (A\bz)_j}{(A\bc)_j}, &\quad for $j: (A\bc)_j < 0$,
\vspace*{3pt}\cr
\displaystyle 0 \leq b_j - (A\bz)_j, &\quad for $j: (A\bc)_j = 0$} \right\},
\end{eqnarray*}
where in the last step, we have divided the components into three
categories depending on whether $(A\bc)_j \gtreqless0$, since this
affects the direction of the inequality (or whether we can divide at
all). Since $\boldeta^T \by$ is the same quantity for all $j$,
it must be at least the maximum of the lower bounds, which is
$\V ^-(\bz)$, and no more than the minimum of the upper bounds, which is $\V ^+(\bz)$.
\end{pf}

Lemma \ref{lem:conditional} tells us that
%
\begin{equation}
\bigl[\boldeta^T \by\mid\{A\by\leq\bb\} \bigr] \overset{d} {=}
\bigl[\boldeta^T \by\mid\bigl\{\V^-(\bz) \leq\boldeta^T
\by\leq\V^+(\bz), \V^0(\bz) \geq0\bigr\} \bigr].
\end{equation}
Since $\V^+(\bz),\V^-(\bz), \V^0(\bz)$ are independent of
$\boldeta^T \by$, they behave as ``fixed'' quantities. Thus,
$\boldeta^T\by$ is conditionally like a normal random variable,
truncated to be between $\V^-(\bz)$ and $\V^+(\bz)$. We would like
to be able to say
\[
\mbox{``}\boldeta^T \by\mid\{A\by\leq\bb\} \sim \mathrm{TN}\bigl(
\boldeta^T \mathbfu, \sigma^2 \boldeta^T
\Sigma\boldeta, \V^-(\bz), \V^+(\bz)\bigr)\mbox{,''}
\]
but this is technically incorrect, since the distribution on the
right-hand side changes with $\bz$. By conditioning on the value of
$\bz$, $\boldeta^T \by\mid\{A\by\leq\bb, \bz=\bz_0\}$ is a
truncated normal. We can then use the probability integral transform to
obtain a statistic $F^\bz(\boldeta^T\by)$ that has a $\unif(0,1)$
distribution for any value of $\bz$. Hence, $F^\bz(\boldeta^T\by)$
will also have a $\unif(0,1)$ distribution marginally over $\bz$. We
make this precise in the next theorem.

%
\begin{theorem}
\label{thm:truncated-gaussian-pivot}
Let $F_{\mu, \sigma^2}^{[a, b]}$ denote the CDF of a $N(\mu, \sigma
^2)$ random variable truncated to the interval $[a, b]$, that is,
%
\begin{equation}
F_{\mu, \sigma^2}^{[a, b]}(x) = \frac{\Phi((x-\mu)/\sigma) - \Phi
((a-\mu)/\sigma)}{\Phi((b-\mu)/\sigma) - \Phi((a-\mu)/\sigma)}, \label{eq:U}
\end{equation}
where $\Phi$ is the CDF of a $N(0, 1)$ random variable. Then
%
\begin{equation}
F_{\boldeta^T\mathbfu, \boldeta^T \Sigma\boldeta}^{[\V^-(\bz),
\V
^+(\bz)]}\bigl(\boldeta^T \by\bigr)\mid\{A\by
\leq\bb\} \sim\unif(0,1), \label{eq:pivot}
\end{equation}
where $\V^-$ and $\V^+$ are defined in \eqref{eq:v_minus} and \eqref
{eq:v_plus}. Furthermore,
\[
\bigl[\boldeta^T \by\mid A\by\leq\bb, \bz= \bz_0 \bigr]
\sim \mathrm{TN}\bigl(\boldeta^T\mathbfu, \sigma^2 \llVert\boldeta
\rrVert^2, \V^-(\bz_0), \V^+(\bz_0)\bigr).
\]
\end{theorem}

\begin{pf}
First, apply Lemma \ref{lem:conditional}:
\begin{eqnarray*}
\bigl[\boldeta^T \by\mid A\by\leq\bb, \bz= \bz_0 \bigr]
&\overset{d} {=}& \bigl[\boldeta^T \by\mid\V^-(\bz) \leq
\boldeta^T\by\leq\V^+(\bz), \V^0(\bz)\geq0, \bz=
\bz_0 \bigr]
\\
&\overset{d} {=}& \bigl[\boldeta^T \by\mid\V^-(\bz_0)
\leq\boldeta^T\by\leq\V^+(\bz_0), \V^0(
\bz_0)\geq0, \bz= \bz_0 \bigr].
\end{eqnarray*}
The only random quantities left are $\boldeta^T \by$ and $\bz$. Now
we can eliminate $\bz= \bz_0$ from the condition using independence:
\begin{eqnarray*}
\bigl[\boldeta^T \by\mid A\by\leq\bb, \bz= \bz_0 \bigr]
&\overset{d} {=} &\bigl[\boldeta^T \by\mid\V^-(\bz_0)
\leq\boldeta^T\by\leq\V^+(\bz_0) \bigr]
\\
&\sim& \mathrm{TN}\bigl(\boldeta^T\mathbfu, \sigma^2 \llVert
\boldeta\rrVert^2, \V^-(\bz_0), \V^+(\bz_0)
\bigr).
\end{eqnarray*}

Letting $F^\bz(\boldeta^T \by) \equiv F_{\boldeta^T\mathbfu, \boldeta
^T \Sigma\boldeta}^{[\V^-(\bz), \V^+(\bz)]}(\boldeta^T \by)$,
we can apply the probability integral transform to the above result to obtain
\begin{eqnarray*}
\bigl[F^{\bz}\bigl(\boldeta^T \by\bigr) \mid A\by\leq\bb,
\bz=\bz_0 \bigr] &\overset{d} {=} &\bigl[F^{\bz_0}\bigl(
\boldeta^T \by\bigr) \mid A\by\leq\bb, \bz=\bz_0 \bigr]
\\
&\sim&\unif(0,1).
\end{eqnarray*}
If we let $p_X$ denote the density of a random variable $X$ given $\{
A\by\leq\bb\}$, what we have just shown is that
\[
p_{F^\bz(\boldeta^T \by) \mid\bz}(t \mid\bz_0) \equiv\frac
{p_{F^\bz
(\boldeta^T \by), \bz}(t, \bz_0)}{p_{\bz}(\bz_0)} =
1_{[0,1]}(f)
\]
for any $\bz_0$. The desired result now follows by integrating over
$\bz_0$:
\begin{eqnarray*}
p_{F^\bz(\boldeta^T \by)}(t) &=& \int p_{F^\bz(\boldeta^T \by)\mid
\bz
}(t\mid\bz_0)
p_{\bz}(\bz_0) \,d\bz_0
\\
&=& \int1_{[0,1]}(t) p_{\bz}(\bz_0) \,d
\bz_0
\\
&=& 1_{[0,1]}(t).
\end{eqnarray*}\upqed
\end{pf}

\subsection{Conditioning on a union of polyhedra}
\label{sec:union}

We have just characterized the distribution of $\boldeta^T \by$,
conditional on $\by$ falling into a single polyhedron $\{ A\by\leq
\bb\}$. We obtain such a polyhedron if we condition on both the model
and the signs $\{ \hat\E= \E, \hat\bs= \bs\}$. If we want to only
condition on the model $\{ \hat\E= \E\}$, then we will have to
understand the distribution of $\boldeta^T \by$, conditional on $\by
$ falling into a union of such polyhedra, that is,
%
\begin{equation}
\boldeta^T \by\big| \bigcup_\bs\{
A_\bs\by\leq\bb_\bs\}.
\end{equation}

%
\begin{figure}

\includegraphics{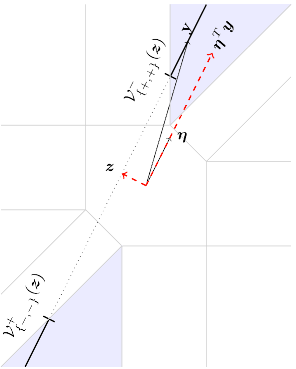}

%
%
%
%
%
%
%
%
%
%
%
\caption{When we take the union over signs, the conditional
distribution of $\boldeta^T \by$ is truncated to a union of disjoint
intervals. In this case, the Gaussian is truncated to the set $(-\infty
, \V^+_{\{-, -\}}(\bz)] \cup[\V^-_{\{+, +\}}(\bz), \infty)$.}
\label{fig:union}
\end{figure}

As Figure~\ref{fig:union} makes clear, the argument proceeds exactly
as before, except that $\boldeta^T \by$ is now truncated to a union
of intervals, instead of a single interval. There is a $\V^-$ and a
$\V^+$ for each possible sign pattern $\bs$, so we index the
intervals by the signs. This leads immediately to the next theorem,
whose proof is essentially the same as that of Theorem \ref
{thm:truncated-gaussian-pivot}.

%
\begin{theorem}
\label{thm:union}
Let $F_{\mu, \sigma^2}^S$ denote the CDF of a $N(\mu, \sigma^2)$
random variable truncated to the set $S$. Then
%
\begin{equation}
F_{\boldeta^T\mathbfu, \boldeta^T\Sigma\boldeta}^{\bigcup_{\bs}
[\V
_{\bs}^-(\bz), \V_{\bs}^+(\bz)]}\bigl(\boldeta^T \by\bigr)\big|
\bigcup_\bs\{ A_\bs\by\leq
\bb_\bs\} \sim\unif(0,1), \label{eq:minimal-pivotal-quantity}
\end{equation}
where $\V_{\bs}^-(\bz)$ and $\V_{\bs}^+(\bz)$ are defined in
\eqref{eq:v_minus} and \eqref{eq:v_plus} and $A = A_\bs$ and $b =
b_\bs$.
\end{theorem}

\section{Post-selection intervals for regression coefficients}
\label{sec:lasso}

In this section, we combine the characterization of the lasso selection
event in Section~\ref{sec:lasso-selection} with the results about the
distribution of a Gaussian truncated to a polyhedron (or union of
polyhedra) in Section~\ref{sec:polyhedra} to form post-selection
intervals for lasso-selected regression coefficients. The key link is
that the lasso selection event can be expressed as a union of polyhedra:
\begin{eqnarray*}
\{\hat\E= \E\} &=& \bigcup_{\bs\in\{-1, 1\}^{\llvert \E
\rrvert }} \{\hat\E= \E, \hat
\bs= \bs\}
\\
&=& \bigcup_{\bs\in\{-1, 1\}^{\llvert \E\rrvert }} \bigl\{A(\E, \bs)\by
\leq\bb(\E,
\bs)\bigr\},
\end{eqnarray*}
where $A(\E, \bs)$ and $\bb(\E, \bs)$ are defined in Theorem \ref
{thm:lasso_partition}. Therefore, conditioning on selection is the same
as conditioning on a union of polyhedra, so we can apply the framework of Section~\ref{sec:polyhedra}.

Recall that our goal is to form confidence intervals for $\beta^\E_j
= \be_j^T X_\E^\dagger\mathbfu$, with $(1-\alpha)$-coverage
conditional on $\{\hat\E= \E\}$. Taking $\boldeta= (X_\E^\dagger
)^T \be_j$, we can use Theorem \ref{thm:union} to obtain
\[
F^{\bigcup_\bs[\V_\bs^-(\bz), \V_\bs^+(\bz)]}_{\beta^\E_j,
\sigma^2 \llVert\boldeta\rrVert^2}\bigl(\boldeta^T \by\bigr)\mid\{\hE=
\E\} \sim\unif(0,1).
\]
This gives us a test statistic for testing any hypothesized value of
$\bbeta^\E_j$. We can invert this test to obtain a confidence set
%
\begin{equation}
C^\E_j \equiv\biggl\{ \beta^\E_j:
\frac{\alpha}{2} \leq F^{\bigcup_\bs[\V_\bs^-(\bz), \V_\bs^+(\bz
)]}_{\beta^\E_j,
\sigma^2 \llVert\boldeta\rrVert^2}\bigl(
\boldeta^T \by\bigr) \leq1-\frac{\alpha}{2} \biggr\}.
\label{eq:conf_int}
\end{equation}
In fact, the set $C^\E_j$ is an {\em interval}, as formalized in the
next result.

%
\begin{theorem}\label{thm:conf_int}
Let $\boldeta= (X_\E^\dagger)^T \be_j$. Let $L$ and $U$ be the
(unique) values satisfying
\begin{eqnarray*}
F_{L, \sigma^2 \llVert\boldeta\rrVert^2}^{\bigcup_\bs
[\V^-_\bs(\bz), \V
^+_\bs(\bz)]}\bigl(\boldeta^T \by\bigr) &=& 1-
\frac{\alpha}{2},\qquad F_{U, \sigma
^2 \llVert\boldeta\rrVert^2}^{\bigcup_\bs[\V^-_\bs(\bz
), \V^+_\bs(\bz
)]}\bigl(\boldeta^T \by
\bigr) = \frac{\alpha}{2}.
\end{eqnarray*}
Then $[L, U]$ is a $(1-\alpha)$ confidence interval for $\beta^\E
_j$, conditional on $\{ \hat\E= \E\}$, that is,
%
\begin{equation}
\label{eq:coverage} \Pp\bigl( \beta^\E_j \in[L, U]\mid\hat
\E= \E\bigr) = 1-\alpha.
\end{equation}
\end{theorem}

\begin{pf}
By construction, $\Pp_{\beta^\E_j}(\beta^\E_j \in C^\E_j | \hat
\E= \E) = 1-\alpha$, where $C^\E_j$ is defined in \eqref
{eq:conf_int}. The claim is that the set $C^\E_j$ is in fact the
interval $[L, U]$. To see this, we need to show that the test statistic
$F_{L, \sigma^2 \llVert\boldeta\rrVert^2}^{\bigcup_\bs
[\V^-_\bs(\bz), \V
^+_\bs(\bz)]}(\boldeta^T \by)$ is monotone decreasing in $\beta^\E
_j$ so that it crosses $1-\frac{\alpha}{2}$ and $\frac{\alpha}{2}$
at unique values. This follows from the fact that the truncated
Gaussian distribution has monotone likelihood ratio in the mean
parameter. See the \hyperref[appendix:monotone]{Appendix} for details.
\end{pf}

Alternatively, we could have conditioned on the signs, in addition to
the model, so that we would only have to condition on a
single polyhedron. We also showed in Section~\ref{sec:polyhedra} that
\[
F^{[\V^-_\bs(\bz), \V^+_\bs(\bz)]}_{\bbeta^\E_j, \sigma^2
\llVert\boldeta\rrVert^2}\bigl(\boldeta^T \by\bigr)\mid\{\hE=
\E, \hat\bs= \bs\} \sim\unif(0,1).
\]
Inverting this statistic will produce intervals that have $(1-\alpha)$
coverage conditional on $\{\hE= \E, \hat\bs= \bs\}$, and hence
$(1-\alpha)$ coverage conditional on $\{\hE= \E\}$. However, these
intervals will be less efficient; they will in general be wider.
However, one may be willing to sacrifice statistical efficiency for
computational efficiency. Notice that the main cost in computing
intervals according to Theorem \ref{thm:conf_int} is determining the
intervals $[\V^-_\bs(\bz), \V^+_\bs(\bz)]$ for each $\bs\in\{
-1, 1\}^{\llvert \E\rrvert }$. The number of such sign
patterns is $2^{\llvert \E\rrvert }$.
While this might be feasible when $\llvert \E\rrvert $ is
small, it is
not feasible when we select hundreds of variables. Conditioning on the
signs means that we only have to compute the interval $[\V^-_\bs(\bz
), \V^+_\bs(\bz)]$ for the sign pattern $\bs$ that was actually observed.

Figure~\ref{fig:ci_comparison} shows the tradeoff in statistical
efficiency. When the signal is strong, as in the left-hand plot, there
is virtually no difference between the intervals obtained by
conditioning on just the model, or the model and signs. On the other
hand, in the right-hand plot, we see that we can obtain very wide
intervals when the signal is weak. The widest intervals are for actual
noise variables, as expected.

%
\begin{figure}

\includegraphics{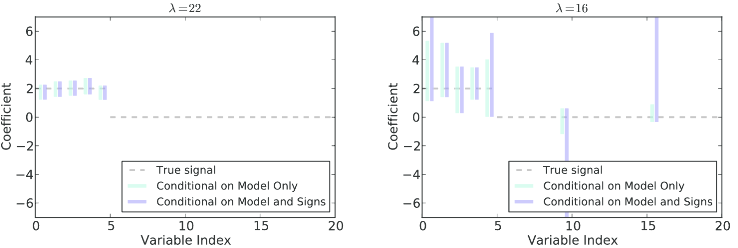}

\caption{Comparison of the confidence intervals by conditioning on the
model only (statistically more efficient, but computationally more
expensive) and conditioning on both the model and signs (statistically
less efficient, but computationally more feasible). Data were simulated
for $n=25$, $p=50$, and 5 true nonzero coefficients; only the first 20
coefficients are shown. (Variables with no intervals are included to
emphasize that inference is only on the selected variables.)
Conditioning on the signs in addition to the model results in no loss
of statistical efficiency when the signal is strong (left) but is
problematic when the signal is weak (right).}
\label{fig:ci_comparison}
\end{figure}

To understand why post-selection intervals are sometimes very wide,
notice that when a truncated Gaussian random variable $Z$ is close to
the endpoints of the truncation interval $[a, b]$, there are many means
$\mu$ that would be consistent with that observation---hence, the
wide intervals. Figure~\ref{fig:intervals} shows confidence intervals
for $\mu$ as a function of $Z$. When $Z$ is far from the endpoints of
the truncation interval, we basically recover the nominal OLS intervals
(i.e., not adjusted for selection).

%
\begin{figure}

\includegraphics{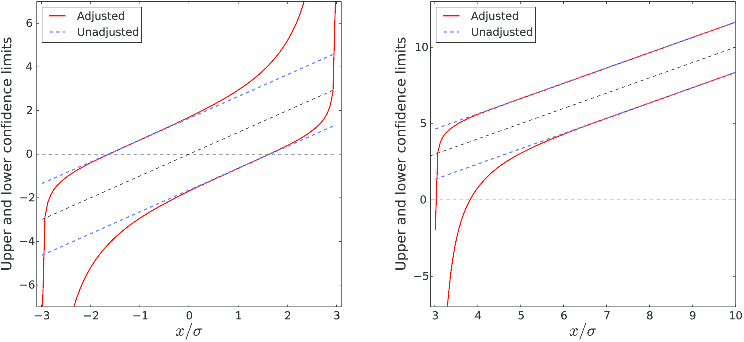}

\caption{Upper and lower bounds of 90\% confidence intervals for $\mu
$ based on a single observation $x/\sigma\sim \mathrm{TN}(0, 1, -3, 3)$. We see
that as long as the observation $x$ is roughly $0.5\sigma$ away from
either boundary, the size of the intervals is comparable to the
unadjusted OLS confidence interval.}
\label{fig:intervals}
\end{figure}

The implications are clear. When the signal is strong, $\boldeta^T \by
$ will be far from the endpoints of the truncation region, so we obtain
the nominal OLS intervals. On the other hand, when a variable just
barely entered the model, then $\boldeta^T \by$ will be close to the
edge of the truncation region, and the interval will be wide.

\subsection{Optimality}

We have derived a confidence interval $C^\E_j$ whose conditional
coverage, given $\{\hat\E= \E\}$, is at least $1-\alpha$. The fact
that we have found such an interval is not remarkable, since many such
intervals have this property. However, given two intervals with the
same coverage, we generally prefer the shorter one.
This problem is considered in \citet{fithian2014optimal} where it is
shown that $C^\E_j$ is,\vspace*{1pt} with one small tweak, the shortest interval
among all {\em unbiased} intervals with $1-\alpha$ coverage.

An {\em unbiased} interval $C$ for a parameter $\theta$ is one which
covers no other parameter $\theta'$ with probability more than
$1-\alpha$, that i,
%
\begin{equation}
\Pp_{\theta}\bigl(\theta' \in C\bigr) \leq1-\alpha\qquad
\mbox{for all }\theta, \theta' \neq\theta.
\end{equation}
Unbiasedness is a common restriction to ensure the existence of an optimal interval
[\citet{TSH}]. The\vspace*{2pt} shortest unbiased interval for  $\beta^\E_j$,
among all intervals with conditional $1-\alpha$ coverage, resembles to the
interval $[L, U]$ in Theorem \ref{thm:conf_int}. There, the critical values $L$ and
$U$ were chosen symmetrically so that the pivot has $\alpha/2$ area in either tail.
However, it may be possible to obtain a shorter interval on average by allocating
the a probability unequally between the two tails. Theorem 5 of \citet
{fithian2014optimal} provides a general formula for obtaining shortest unbiased intervals in exponential families.

\section{Data example}
\label{sec:examples}

We apply our post-selection intervals to the diabetes data set from
\citet{efron2004least}. Since $p < n$ for this data set, we can estimate
$\sigma^2$ using the residual sum of squares from the full regression
model with all $p$
predictors. After standardizing all variables, we chose $\lambda$
according to the strategy in \citet{negahban2012unified}, $\lambda= 2
\Expect(\llVert X^T\varepsilon\rrVert_{\infty})$. This
expectation was computed by
simulation, where $\varepsilon\sim N(0, \hat\sigma^2)$, resulting in
$\lambda\approx190$. The lasso selected four variables:
\texttt{BMI}, \texttt{BP}, \texttt{S3} and \texttt{S5}.

The post-selection intervals are shown in Figure~\ref{fig:diabetes},
alongside the nominal confidence intervals produced by fitting OLS to
the four selected variables, ignoring selection. The nominal intervals
do not have $(1-\alpha)$ coverage conditional on the model and are not
valid post-selection intervals. Also depicted are the confidence
intervals obtained by data splitting, as discussed in Section~\ref
{sec:goals}. This is a competitor method that also produces valid
confidence intervals conditional on the model. The lasso selected the
same four variables on half of the data, and then nominal intervals for
these four variables using OLS on the other half of the data.

%
\begin{figure}

\includegraphics{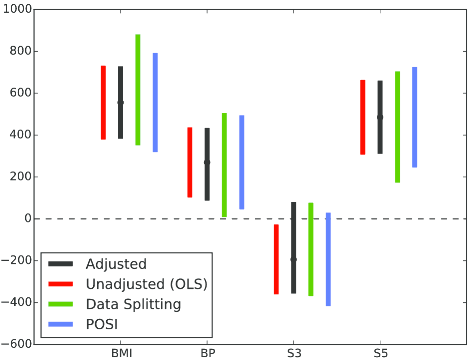}

\caption{Inference for the four variables selected by the lasso
($\lambda= 190$) on the diabetes data set. The point estimate and
adjusted confidence intervals using the approach in Section~\protect
\ref{sec:lasso} are shown in black. The OLS intervals, which ignore
selection, are shown in red. The green lines show the intervals
produced by splitting the data into two halves, forming the interval
based on only half of the data. The blue line corresponds to the POSI
method of \citet{berk2013posi}.}
\label{fig:diabetes}
\end{figure}

We can make two observations from Figure~\ref{fig:diabetes}.
\begin{longlist}[2.]
\item[1.] The adjusted intervals provided by our method essentially
reproduces the OLS intervals for the strong effects, whereas data
splitting intervals are wider by a factor of $\sqrt{2}$ (since only
$n/2$ observations are used in the inference). For this dataset, the
POSI intervals are $1.36$ times wider than the OLS intervals. For all
the variables, our method produces the shortest intervals among the
methods that control selective type 1 error.

\item[2.] One variable, \texttt{S3}
which would have been deemed
significant using the OLS intervals, is no longer significant after
accounting for selection. Data splitting, our selection-adjusted
intervals, and POSI intervals conclude that
\texttt{S3} is not
significant. This demonstrates that taking model selection into account
can have substantive impacts on the conclusions.\vadjust{\eject}
\end{longlist}

\section{Extensions}
\subsection{Estimation of \texorpdfstring{$\sigma^2$}{sigma2}}
\label{sec:estimate:sigma}

The above results rely on knowing $\sigma^2$ or at least having a good
estimate of it. If $n > p$, then the variance $\hat\sigma^2$ of the
residuals from fitting the full model is a consistent estimator and in
general can be substituted for $\sigma^2$ to yield asymptotically
valid confidence intervals. Formally, the condition is that the pivot
is smooth with respect to $\sigma$. Geometrically speaking, the upper
and lower truncation limits $\V^+$ and $\V^-$ must be well-separated
(with high probability). We refer the interested reader to Section~2.3
in \citet{tian2015asymptotics} for details.

In the setting where $p > n$, obtaining an estimate of $\sigma^2$ is
more challenging, but if the pivot satisfies a monotonicity property,
plugging in an overestimate of the variance gives conservative
confidence intervals. We refer the reader to Theorem~11 in \citet
{tibshirani2015uniform} for details.

\subsection{Elastic net}
One problem with the lasso is that it tends to select one variable out
of a set of correlated variables, resulting in estimates that are
unstable. One way to stabilize them is to add an $\ell_2$ penalty to
the lasso objective, resulting in the elastic net [\citet
{zou2005regularization}]:
%
\begin{eqnarray}
\tilde\bbeta&=& \mathop{\operatorname{argmin}}\limits
_{\bbeta} \frac{1}{2} \llVert
\by-X \bbeta\rrVert_2^2 +\lambda\llVert\bbeta\rrVert
_1+ \frac{\gamma}{2} \llVert\bbeta\rrVert_2
^2. \label{eq:elastic-net}
\end{eqnarray}
%
Using a nearly identical argument to Lemma \ref{lem:equiv_sets}, we
see that $\{ \hat\E= \E, \hat\bs= \bs\}$ if and only if there
exist $\tilde\bw$ and $\tilde\bu$ satisfying
\begin{eqnarray*}
\bigl(X_\E^T X_\E+ \gamma I\bigr) \tilde\bw-
X_\E^T \by+ \lambda\bs&=& 0,
\\
X_{-\E}^T (X_\E\tilde\bw- \by) + \lambda\tilde
\bu&=& 0,
\\
\sign(\tilde\bw) &=& \bs,
\\
\llVert\tilde\bu\rrVert_\infty&<& 1.
\end{eqnarray*}
These four conditions differ from those of Lemma \ref{lem:equiv_sets}
in only one respect: $X_\E^T X_\E$ in the first expression is
replaced by $X_\E^T X_\E+ \gamma I$. Continuing the argument of
Section~\ref{sec:lasso-selection}, we see that the selection event can
be rewritten
%
\begin{equation}
\{ \hat\E= \E, \hat\bs= \bs\} = \biggl\{ \pmatrix{ \tilde A_0(\E,
\bs)
\vspace{3pt}\cr
\tilde A_1(\E, \bs) } \by< \pmatrix{ \tilde\bb_0(
\E, \bs)
\vspace{3pt}\cr
\tilde\bb_1(\E, \bs) } \biggr\}, \label{eq:enet_A_b}
\end{equation}
where\vspace*{1pt} $\tilde A_k$ and $\tilde\bb_k$ are analogous to $A_k$ and $\bb
_k$ in Proposition~\ref{prop:A_b}, except replacing $(X_\E^T X_\E
)^{-1}$ by $(X_\E^T X_\E+ \gamma I)^{-1}$ everywhere\vspace*{1pt} it appears.
Notice that $(X_\E^T X_\E)^{-1}$ appears explicitly in $A_1$ and $\bb
_1$, and also\vspace*{1pt} implicitly in $A_0$ and $\bb_0$, since $P_\E$ and
$(X_\E^T)^\dagger$ both depend on $(X_\E^T X_\E)^{-1}$.

Now that we have rewritten the selection event in the form $\{A\by\leq
\bb\}$, we can once again apply the framework in Section~\ref
{sec:polyhedra} to obtain a test for the elastic net conditional on
this event.
\section{Conclusion}

Model selection and inference have long been regarded as conflicting
goals in linear regression. Following the lead of \citet
{berk2013posi}, we have proposed a framework for post-selection
inference that {\em conditions on which model was selected}, that is,
the event $\{\hat\E= \E\}$. We characterize this event for the lasso
and derive optimal and exact confidence intervals for linear contrasts
$\boldeta^T \mathbfu$, conditional on $\{\hat\E= \E\}$. With this
general framework, we can form post-selection intervals for regression
coefficients, equipping practitioners with a way to obtain ``valid''
intervals even after model selection.



%
\begin{appendix}\label{appendix:monotone}
%
\section*{Appendix: Monotonicity of $F$}

%
\begin{lemma}
Let\vspace*{1pt} $F_{\mu}(x):= F_{\mu, \sigma^2}^{[a,b]}(x)$ denote the
cumulative distribution function of a truncated Gaussian random
variable, as defined as in \eqref{eq:U}. Then $F_\mu(x)$ is monotone
decreasing in $\mu$.
\end{lemma}

\begin{pf}
First, the truncated Gaussian distribution with CDF $F_{\mu}:= F_{\mu
, \sigma^2}^{[a,b]}$ is a natural exponential family in $\mu$, since
it is just a Gaussian with a different base measure. Therefore, it has
monotone likelihood ratio in $\mu$. That is, for all $\mu_1 > \mu_0$
and $x_1 > x_0$:
\[
\frac{f_{\mu_1}(x_1)}{f_{\mu_0}(x_1)} > \frac{f_{\mu
_1}(x_0)}{f_{\mu_0}(x_0)},
\]
where $f_{\mu_i}:= dF_{\mu_i}$ denotes the density. (Instead of
appealing to properties of exponential families, this property can also
be directly verified.)

This implies
\begin{eqnarray*}
f_{\mu_1}(x_1) f_{\mu_0}(x_0) &>&
f_{\mu_1}(x_0) f_{\mu_0}(x_1), \qquad
x_1 > x_0.
\end{eqnarray*}
Therefore, the inequality is preserved if we integrate both sides with
respect to $x_0$ on $(-\infty, x)$ for $x < x_1$. This yields
\begin{eqnarray*}
\int_{-\infty}^{x} f_{\mu_1}(x_1)
f_{\mu_0}(x_0) \,dx_0 &>& \int_{-\infty}^{x}
f_{\mu_1}(x_0) f_{\mu_0}(x_1) \,dx_0,\qquad  x < x_1,
\\
f_{\mu_1}(x_1) F_{\mu_0}(x) &>& f_{\mu_0}(x_1)
F_{\mu_1}(x), \qquad x < x_1.
\end{eqnarray*}
Now we integrate both sides with respect to $x_1$ on $(x, \infty)$ to obtain
\begin{eqnarray*}
\bigl(1 - F_{\mu_1}(x)\bigr) F_{\mu_0}(x) &>& \bigl(1 -
F_{\mu_0}(x)\bigr) F_{\mu_1}(x)
\end{eqnarray*}
which establishes $F_{\mu_0}(x) > F_{\mu_1}(x)$ for all $\mu_1 > \mu_0$.
\end{pf}
\end{appendix}


\section*{Acknowledgements}
We thank Will Fithian, Sam Gross and Josh Loftus for helpful comments
and discussions. In particular, Will Fithian provided insights that led
to the geometric intuition of our procedure shown in Figure~\ref{fig:polyhedron}.


%

\printaddresses
\end{document}